\documentclass{article}

\usepackage{amsfonts}
\usepackage{amssymb}
\usepackage{amsthm}

\usepackage{amsmath}

\usepackage{selectp}

\usepackage{xcolor}

\newtheorem{theorem}{Theorem}
\newtheorem{proposition}[theorem]{Proposition}
\newtheorem{corollary[theorem]}{Corollary}

\newtheorem{remark}[theorem]{Remark}
\newtheorem{example}[theorem]{Example}

\newcommand{\R}{\mathbb R}
\newcommand{\N}{\mathbb N}

\newcommand{\cP}{\mathcal P}
\newcommand{\cH}{\mathcal H}
\newcommand{\cD}{\mathcal D}

\newcommand{\commentout}[1]{{}}

\makeatletter
\let\@fnsymbol\@arabic
\makeatother
\begin{document}

\title{Packing and Hausdorff measures of Cantor sets associated with series}

\author{ K. Hare \thanks{Partially supported by NSERC 44597} \and F. Mendivil\thanks{Partially supported by NSERC}\and L. Zuberman\thanks{Partially supported by CONICET }}

\maketitle

\begin{center} \it \small
$^{1}$Department of Pure Math, University of Waterloo, Wateloo, Ont., Canada \\%Partially supported by NSERC 44597.
$^{2}$Department of Mathematics and Statistics, Acadia University, Wolfville, Nova Scotia, Canada\\ 
$^{3}$Department of Mathematics, Universidad Nacional de Mar del Plata, Mar del Plata, Argentina.

\end{center}

\begin{abstract}
We study a generalization of Mor\'an's sum sets, obtaining information about the $h$-Hausdorff and $h$-packing measures of these sets and certain of their subsets.
\end{abstract}


%------------------------------------------------------------

\section{Introduction}

In  \cite{Mo89} Mor\'an introduced the notion of a sum set,
\[
C_{a}=\left\{ \sum_{i=1}^{\infty }\varepsilon _{i}a_{i}:\varepsilon
_{i}=0,1\right\} ,
\]
the set of all possible subsums of the series $\sum a_{n}$ where $a=(a_{n})$
is a sequence of vectors in $\mathbb{R}^{p}$ with summable norms. The
classical Cantor middle-third set is one example with $a_i=3^{-i}2$.  Assuming a suitable
separation condition, in  \cite{Mo94} Mor\'an related the $h$-Hausdorff measure of $C_{a}$ to
the quantities $R_{n}=\sum_{i>n}\left\Vert a_{i}\right\Vert $.

In this paper, we generalize Mor\'an's sum set notion to permit a greater diversity in the geometry. (See (\ref{eq:gen}) for the
definition of the generalization.) For example, our
generalization includes Cantor-like sets in $\mathbb{R}$ which have the
property that the Cantor intervals of a given level (but not necessarily the gaps) are all of the same
length. Moreover, unlike Mor\'an's sets, our generalized sum sets can
have Hausdorff dimension greater than one.

We obtain the analogue of Mor\'an's results on $h$-Hausdorff measures for these generalized sum sets and prove dual results for $h$-packing measures.
We show that for any of these sum sets there is a doubling dimension function $h$ for which the sum set has both finite and positive 
$h$-Hausdorff and $h$-packing measure. 
We give formulas for the Hausdorff and packing dimensions, and show that given any $\alpha$ less than the Hausdorff dimension (or $\beta$ less than the packing dimension) there is a sum subset that has Hausdorff dimension $\alpha$ (or packing dimension $\beta$). 
In fact, there is even a sum subset with both Hausdorff dimension $\alpha$ and packing dimension $\beta$ provided $\alpha/\beta$ is dominated by the ratio of the Hausdorff dimension to the packing dimension of the original set.
Furthermore, if the Hausdorff and/or packing measure is finite and positive (in the corresponding dimension), then we can choose this sum subset to have finite and positive Hausdorff and/or packing measure.

\section{Preliminaries}

Let $s_n > 0$ with $\sum_n s_n < \infty$.
Fix $N \in \N$ and for each $n \in \N$ let  the \emph{$n$th digit set} $\cD^n = \{ 0 = d^n_1, d^n_2, \ldots, d^n_N \} \subset \R^p$ be given.
We define $C_{s,\cD}$ by
\begin{equation}
 \label{eq:gen}
    C_{s,\cD} = \{ \sum_{i=1}^{\infty} s_i b_i : b_i \in \cD^i \},
\end{equation}
the set of all possible sums with choices drawn from $\cD^n$ and scaled by $s_n$. 
Mor\'an's sum set is the special case when $s_i = ||a_i||$, $N=2$ and $ \cD^i = \{0,a_i/||a_i||\}$. 
This generalized sum set is the main object of study in this paper.

For each $n$ define
\[
    \kappa_n = \max \{ \| d^n_i - d^n_j \| : 0 \le i, j \le N, i \ne j \}
\]
and
\[
  \tau_n = \min \{ \| d^n_i - d^n_j \| : 0 \le i, j \le N, i \ne j \}.
\]
In Mor\'an's case, $\kappa_n = \tau_n =1$.
We assume that $\kappa := \sup_n \kappa_n < \infty$, as well as $\tau := \inf \tau_n > 0$;
the intent is that the sequence $s_n$ controls the decay rate, not the (possibly varying) geometry of the digit sets $\cD^n$.
In addition, we assume the rapid decay condition
\begin{equation} \label{eq:kappacondition}
    \sup_n \frac{\kappa R_n}{\tau s_n} = M < 1,
\end{equation}
where $R_n = \sum_{i > n} s_i$.  This is the analogue of Mor\'an's separation condition. The quantity $R_n$ is very important for describing the geometry of $C_{s,\cD}$. 

In certain situations where we have precise information about the geometry of $\cD^n$, it is possible to assume something weaker
than (\ref{eq:kappacondition}) and still have a suitable separation property to allow for dimensions to be calculated; see Example \ref{example:alldimensions}.

\begin{example}\label{example:Cantor}
\begin{enumerate}
\item  A very simple example is the classical Cantor set with $s_n=2\cdot 3^{-n}$ and $\cD=\{0,1\}$. 
\item Consider a finite set $\cD\subset \R^p$, a real number $r<  d/(2D)$ (where $d=\min \tilde \cD$, $D=\max\tilde \cD$ and $\tilde \cD =\{\|d-d'\|:d,d'\in\cD,d\neq d'\}$), a matrix $O\in\R^{p\times p}$ orthogonal and the contractions $S_d(x)=rO(x+d)$. The attractor of this IFS is $C_{s,D}$ with $s_n=r^n$ and $\cD^n=O^n\cD$. 
\end {enumerate}
\end{example}

We now examine some basic properties of $C_{s,\cD}$.
First we argue that $C_{s,\cD}$ is a compact and perfect set.
To do this, let $\Xi = \{1,2,\ldots, N \}^\N$ with the product topology induced by the discrete topology on each factor.
 Further, for $n \in \N$ let $\Xi^n = \{ 1,\ldots, N\}^n$.
We note that $\Xi$ is a totally disconnected, perfect metric space.  Define the function $\Phi:\Xi \to \R^p$ by
\[
     \Phi(\sigma) = \sum_i s_i d^i_{\sigma_i}.
\]
Then the range of $\Phi$ is $C_{s,\cD}$.
Since $\Xi$ is compact and perfect, we need only show that $\Phi$ is continuous and injective to show that $C_{s,\cD}$ is compact and perfect.
Let $\Phi_n:\Xi \to \R^p$ be defined by $\Phi_n(\sigma) = \sum_{i \le n} s_i d^i_{\sigma_i}$.
Then $\Phi_n$ is constant on each of the sets $\Xi_\alpha = \{ \sigma \in \Xi :  \sigma_i = \alpha_i, 1 \le i \le n \}$ for any fixed $\alpha \in \Xi^n$.
This means that each $\Phi_n$ is continuous.
Furthermore, $\|\Phi_n(\sigma) - \Phi(\sigma)\|\le \kappa R_n$ and thus
$\Phi_n \to \Phi$ uniformly on $\Xi$ and so $\Phi$ is also continuous.
 Thus $C_{s,\cD}$ is compact.

 If $n$ is the first place where $\sigma$ and $\sigma'$ disagree,
\begin{eqnarray}
     \|\Phi(\sigma) - \Phi(\sigma')\| &=& \| \sum_i s_i ( d^i_{\sigma(i)} - d^i_{\sigma'(i)} )\| \cr
                                                      &\ge& \|s_n ( d^n_{\sigma(n)} - d^n_{\sigma'(n)} )\| - \| \sum_{i > n} s_i (d^i_{\sigma(i)} - d^i_{\sigma'(i)}) \| \cr
                                                      &\ge& s_n \tau - \kappa R_n > 0. \label{eq:separation}
\end{eqnarray}
This means that $\Phi$ is injective and is thus a homeomorphism, so that $C_{s,\cD}$ is also totally disconnected and perfect.

For a given $n \in \N$ and $\sigma \in \Xi^n$, we define
\[
    x_\sigma = \sum_{i \le n}  s_i d^i_{\sigma_i}
\]
and
\[
     C_{\sigma,n} = x_\sigma + \{  \sum_{i>n} s_i b_i : b_i \in \cD^i \}.
\]
Our condition (\ref{eq:kappacondition}) ensures the non-overlapping of the sets $C_{\sigma, n}$.

Using this notation, we see two very important facts.  First,   $C_{\sigma,n} = x_\alpha - x_\sigma + C_{\alpha,n}$ for any $\sigma,\alpha \in \Xi^n$.
That is, for a fixed $n$ the collection of $C_{\sigma,n}$ are all translates of each other.
Secondly, we can decompose $C_{s,\cD}$ into $N^n$ copies of $C_{\sigma,n}$ as
\[
    C_{s,\cD} = \bigcup_{\sigma \in \Xi^n} C_{\sigma,n}   = \{ x_\sigma : \sigma \in \Xi^n \} + C_{1,n},
\]
where by $1 \in \Xi^n$ we mean the element all of whose terms equal to $1$.

  An elementary estimate gives that
\begin{equation} \label{eq:diameter}
     |C_{\sigma,n}|  \le \| \sum_{i > n} s_i b_i - \sum_{i > n} s_i b_i' \| \le \kappa \sum_{i > n} s_i  = \kappa R_n
\end{equation}
where $| C|$ means the diameter of the set $C$.

%%%%%%%%%%%%%%%%%%%%%%%%%%%%%%%%%%%%%%%%%%%
\section{Hausdorff and packing measures}

\label{sec:hausdorff}

We first recall some facts about Hausdorff and packing measures (see \cite{Ro98, Mat95}).  For us, a \emph{dimension function} is a continuous non-decreasing function
$h:[0,\infty) \to [0,\infty)$ with $h(0) = 0$.  It is said to be \emph{doubling} if there is some constant $c > 0$ so that $h(2 x) \le c\, h(x)$ for
all $x>0$.

For two dimension  function $f,g$ we say that $f \prec g$ if
\[
   \lim_{t \to 0+} g(t)/f(t) = 0.
\]

For each $\delta > 0$, a $\delta$-covering of a set $E$ is a countable collection $\{ B_i \}$ of subsets of $\R^p$ with diameters dominated by $\delta$, that is
$|B_i| \le \delta$, and for which $E  \subseteq \cup_i B_i$.  We define
\[
    \cH^h_\delta(E) = \inf \{ \sum_i h(|B_i|) : \{ B_i \} \mbox{ is a $\delta$-covering of $E$} \}
\]
and the \emph{Hausdorff $h$-measure} as
\[
    \cH^h(E) = \lim_{\delta \to 0} \cH^h_\delta(E).
\]
Notice that in the definition of $\cH^h_\delta$  it is sufficient to consider coverings by balls.

Now we turn to the $h$-packing measure $\cP^h$.
A $\delta$-packing of a set $E$ is a disjoint family of open balls $\{ B(x_i, r_i) \}$ with $x_i \in E$ and $r_i \le \delta$.
The $h$-packing pre-measure is given by
\[
    \cP^h_0(E) = \lim_{\delta \to 0} \cP^h_\delta(E)
\]
where
\[
     \cP^h_\delta = \sup \{ \sum_i h(|B_i|) : \{ B_i \} \mbox{ is a $\delta$-packing of $E$} \}.
\]
Unfortunately $\cP^h_0$ is not a measure as it is in general not countably additive.  Thus we need one more step to construct the packing
measure $\cP^h$,
\[
  \cP^h(E) = \inf \{ \sum_i \cP^h_0(E_i) :  E \subset \bigcup_i E_i \}.
\]

The next Theorem gives estimates for the Hausdorff and packing measures of $C_{s,\cD}$.
The first two claims about the Hausdorff measure of $C_{s,\cD}$ are given in \cite{Mo94} for the special case of $\cD^n$ containing two digits.

\begin{theorem}  \label{thm:measureestimates}
Suppose that $h$ is a doubling dimension function.
\begin{enumerate}
    \item If $\liminf  N^n h(\kappa R_n) = \alpha$ then $\cH^h(C_{s,\cD}) \le \alpha$.

    \item If $\liminf  N^n h(\kappa R_n) = \alpha > 0$ then $\cH^h(C_{s,\cD}) > 0$.

    \item If $\limsup  N^n h(\kappa R_n) = \alpha < \infty$ then $\cP^h(C_{s,\cD}) \le N \alpha$.

    \item If $\limsup  N^n h(\kappa R_n)  = \alpha  > 0$ then $\cP^h(C_{s,\cD}) > 0$.

\end{enumerate}
\end{theorem}

\begin{remark}
If $h$ is doubling and $0 < \liminf N^n h( R_n) < \infty$, then $0< \cH^h(C_{s,\cD}) < \infty$ and so $C_{s,\cD}$ is an $h$-Hausdorff set.
Similarly, if $ \limsup N^n h(R_n)$ is positive and finite, then $C_{s,\cD}$ is
an $h$-packing set. Finally, if $\liminf N^n h(R_n)$ is positive and $\limsup N^n h(R_n)$ is finite, then $C_{s,\cD}$ is both an $h$-Hausdorff and an $h$-packing set.
 \end{remark}

\begin{proof}
Item 1) is trivial by considering the covering $C_{I,n}$ for all $I \in \{0,1,\ldots, N \}^n$, which consists of $N^n$ sets all of diameter at most $\kappa R_n$.

To prove the rest of the statements, we will use the fact that there is a Borel measure $\mu$ supported on $C_{s,\cD}$ for which
$\mu(C_{\sigma,n}) = N^{-n}$ for each $\sigma$ and $n$.  This measure is often called the \emph{natural probability measure}.

\medskip

\noindent\emph{2): \quad}  Let $\beta < \alpha$ so that we have $N^n h(\kappa R_n) > \beta$ for all large $n$.  Now
choose $x \in C_{s,\cD}$ and $\delta > 0$ and let $n$ be such that $\kappa R_n < \delta \le \kappa R_{n-1}$.
By a simple modification of Lemma 2 in \cite{Mo89} there is a $q \in \N$ so that the number of $C_{\sigma,n}$ which intersect
$B(x,\delta)$ is less than $q$ (independent of $B$ and $\delta$). (This is where the condition  (\ref{eq:kappacondition}) is used.)  But then we have
\[
   \mu(B(x,\delta)) \le q \mu(C_{\sigma,n})  = q\, N^{-n} < \frac{q}{\beta} h(\kappa R_n) < \frac{q h(\delta)}{\beta}.
\]
By the mass distribution principle (see \cite{FalFG}), we have $\cH^h(C_{s,\cD}) \ge \alpha/q$.

\medskip

\noindent\emph{3) \quad}  Let $\beta > \alpha$ so that we have $N^n h( \kappa R_n) < \beta$ for all large $n$.
Now choose $x \in C_{s,\cD}$ and $\delta > 0$ and let $n$ be such that $\kappa R_n < \delta \le \kappa R_{n-1}$.
We know that $x \in C_{\sigma,n}$ for some $\sigma$ and, since $|C_{\sigma,n}| \le \kappa R_n < \delta$, we have that $C_{\sigma,n} \subseteq B(x,\kappa R_n) \subseteq B(x,\delta)$.
But then
\[
   \mu(B(x,\delta)) \ge \mu(C_{\sigma,n}) = N^{-n} = \frac{N^{-(n-1)}}{N} >  \frac{h( \kappa R_{n-1})}{N \beta}  \ge  \frac{h(\delta)}{N  \beta},
\]
since $h$ is a nondecreasing function.   But then  we have that $$\liminf \mu(B(x,\delta))/h(\delta) \ge ( N \alpha)^{-1}$$ and so  $\cP^h(C_{s,\cD}) \le N \alpha$
by Theorem 3.16 in \cite{C95}.\\

\medskip

\noindent\emph{4) \quad} Let $ 0 < \beta < \alpha$.
Then there are $n_j$ so that $N^{n_j} h(\kappa R_{n_j}) > \beta$ for all $j$. Let
$x \in C_{s,\cD}$ be given.  For any $j$ we have $x \in C_{\sigma_j,n_j}$ for some $\sigma_j$.
By the same simple modification of Lemma 2 in \cite{Mo89}, there is a $q \in \N$ so that for any $\delta > 0$ and any ball $B$ of radius $\delta$, if $m\in \N$ is the smallest value with $\kappa R_m < \delta$ then the number of  $C_{I,m}$ which intersect $B$ is less than $q$ (independent of $B$ and $\delta$).
Let $\delta =  \kappa R_{n_j-1}$, so $\kappa R_{n_j} < \delta = \kappa R_{n_j-1}$.
Then
\[
   \mu(B(x, \kappa R_{n_j}))  \le \mu(B(x,\delta)) \le q\,\mu(C_{\sigma_j,n_j}) = q\, N^{-n_j} <  q\, h(\kappa R_{n_j})/\beta
\]
and thus $\liminf \mu(B(x,\delta))/h(\delta) \le q/\alpha$. By Theorem 3.16 in \cite{C95}, it follows that $\cP^h(C_{s,\cD}) \ge c\alpha/q$, where $c$ is the doubling constant for $h$.
\end{proof}

\begin{remark} \label{rmk:s_n}
Since $\kappa s_{n+1} < \kappa R_n \le M \tau s_n < \kappa s_n$, for any doubling dimension function $h$,
we could instead relate the two quantities, $\liminf N^n h( s_n)$ and $\limsup N^n h( s_n)$, to the $h$-Hausdorff and $h$-packing measure of $C_{s,\cD}$.
\end{remark}

\begin{theorem}  For any sequence $s_n$ and collections of digits $\cD^n$ which satisfy (\ref{eq:kappacondition}), there is a doubling dimension function $h$ for which
$C_{s,\cD}$ is simultaneously both an $h$-Hausdorff set and an $h$-packing set.
\end{theorem}

\begin{proof}

Following  the pattern in \cite[Section 5]{CMMS04}, we define the function $h:[0,\kappa R_0] \to \R$ by $h(0) = 0$ and $h(x) = 1/f^{-1}(x)$ where $f(x)$ is given by
\[
     f(x) = \kappa R_n + \frac{\kappa R_{n+1} - \kappa R_n}{N^{n+1} - N^n} ( x - N^n), \quad   x \in [N^n, N^{n+1}).
\]
Clearly $h$ is non-decreasing and continuous, so we only need to show that $h$ is doubling.
For $x > 0$, let $n,m \in \N$ be such that
$\kappa R_{m+1} < x \le \kappa R_{m} < \kappa R_{n} \le 2 x < \kappa R_{n-1}$.
Then $\kappa R_i \le \tau M s_i \le \kappa \frac{\tau M}{\kappa} R_{i-1}$ for all $i$.
Letting $\theta = \tau M/\kappa < 1$,
\[
      \theta^{n-m} \le \frac{ \kappa R_n}{\kappa R_m} < \frac{2 x}{x} = 2
\]
and so we have $m - n \le  -\ln(2)/\ln(\theta)$.  As $f(N^j) = \kappa R_j$,
\[
     \frac{h(2x)}{h(x)} = \frac{f^{-1}(x)}{f^{-1}(2x)} \le \frac{N^{m+1}}{N^{n-1}} \le N^{2 -\ln(2)/\ln(\theta)},
\]
and so $h$ is doubling.

Since $N^n h(\kappa R_n) = 1$ for all $n$, we have $C_{s,\cD}$ is an $h$-Hausdorff set and an $h$-packing set for this dimension function $h$, as desired.
\end{proof}

The next theorem is a simple consequence of some known results.
However, it shows that the set of dimensional subsets of $C_{s,\cD}$ is an initial segment in the partially ordered set of all doubling dimension functions.

\begin{theorem} \label{thm:comparingHausdorff}
Let $f, h$ be doubling dimension functions and assume $f \prec h$.
\begin{enumerate}
  \item \label{comparingHausdorff}If $0 < \cH^h(C_{s,\cD}) < \infty$, then for any $t > 0$ there is a compact and perfect subset $E \subset C_{s,\cD}$ so that
$\cH^f(E) = t$.
  \item \label{comparingpacking} If $0 < \cP^h(C_{s,\cD}) < \infty$, then for any $t > 0$ there is a compact and perfect subset $E \subset C_{s,\cD}$ so that
$\cP^f(E) = t$.
 \end{enumerate}
\end{theorem}

\begin{proof}
\ref{comparingHausdorff}.
From Theorem 40 in \cite{Ro98}, we have that $\cH^f(C_{s,\cD}) = \infty$.  Then by Theorem 2 in \cite{La67} there is some closed subset
$E' \subset C_{s,\cD}$ for which $\cH^f(E') = t$.  As $E'$ is a closed subset of a perfect set, it is the union of a perfect set $E$ and
a countable set, so $\cH^f(E) = \cH^f(E') = t$ and $E$ is a perfect subset of $C_{s,\cD}$.

\ref{comparingpacking}.
By the same argument as Theorem 40 in \cite{Ro98}, but adapted to packing measures, we have that $\cP^f(C_{s,\cD}) = \infty$.
Now, if we obtain a closed subset
$E' \subset C_{s,\cD}$ for which $\cP^f(E')  = t$, then we find a perfect subset  in a similar way to the case \ref{comparingHausdorff} before.
In \cite{JP95}, Joyce and Preiss proved that if a set has infinite $h$-packing measure (for any given  $h \in \cD$), then the set contains a compact subset with finite $h$-packing measure.
With a simple modification of their proof, (in particular their Lemma 6), we obtain a set of finite packing measure greater than $t$.
%Then, using standard properties of regular, continuous measures, we get a set whose $h$-packing measure is exactly $t$.
By Lyapunov's convexity theorem, there is a subset whose $h$-packing measure is exactly $t$ \cite[Theorem 5.5]{Ru91}.
\end{proof}

We now specialize to the ``usual'' dimension functions $h_s(x) = x^s$ and let $\dim_H$ and $\dim_P$ denote the ``usual'' Hausdorff and packing dimension.
In analogy with the case of a ``cut-out'' Cantor subset of $\R$ (see \cite{BT54,CMMS04,GMS07}, we have the following Proposition.

\begin{proposition} \label{dimensions}
We have that
\[
    \dim_H(C_{s,\cD}) = \liminf \frac{- n \ln(N)}{\ln(s_n)} \quad \mbox{ and } \quad
    \dim_P(C_{s,\cD}) = \limsup \frac{-n \ln(N)}{\ln(s_n)}.
\]
\end{proposition}

\begin{proof}
First, we note that
\[
    \liminf \frac{ -n \ln(N)}{\ln( \kappa R_n)}  = \liminf \frac{ - n \ln(N)}{\ln( R_n)} = \liminf \frac{ -n \ln(N)}{\ln( s_n)},
\]
with a similar equality for the limit superior.

If $\beta > \alpha := \liminf \frac{-n \ln(N)}{\ln( \kappa R_n)}$, then there is a subsequence $(n_j)$ so that $N^{n_j} (\kappa R_{n_j})^{\beta} < 1$. Thus
$\liminf N^n (\kappa R_n)^{\beta} < 1$ and so $\dim_H(C_{s,\cD}) \le \alpha$ by Theorem \ref{thm:measureestimates}.

Conversely, if $\gamma < \alpha$, then for large $n$ we have $N^n (R_n)^{\gamma} > 1$ and thus
$\liminf N^n (R_n)^{\gamma} > 1$ and so $\dim_H(C_{s,\cD}) \ge \alpha$ by Theorem \ref{thm:measureestimates}.

The proof for packing dimension is similar.
\end{proof}

\begin{example} \label{example:alldimensions}
For any $\alpha \in [0,p)$, it is possible to construct a sum set, $C_{s,\cD} \subset \R^p$, with $\dim_H(C_{s,\cD}) = \alpha$.
The simplest way of doing this is to choose $\cD^n = \{ (\epsilon_1, \epsilon_2, \ldots, \epsilon_p) : \epsilon_i \in \{0,1 \} \}$, the set of all corners of a $p$-dimensional unit cube, and set
$s_n = \lambda^n$ where $\lambda = 2^{-p/\alpha}$.  This will generate a self-similar set, $C_{s,\cD}$, that is a product of classical Cantor sets.
The problem is that condition (\ref{eq:kappacondition}) requires that $\lambda < 1/(1+\sqrt{p})$, which does not allow the full range of dimensions (and, in fact, gets worse as $p$ increases).
However, from the simple geometry of this example, we can see that the sets $C_{\sigma,n}$ are non-overlapping provided
$s_n > R_n$. Under this (weaker) assumption,  $C_{s,\cD}$ is a self-similar set satisfying the open set condition and hence its dimensions are as stated in the previous proposition. This separation condition allows for any $\lambda \in [0,1/2)$.

\end{example}

In the case of the ``usual'' dimension functions $h_s$, Theorem \ref{thm:comparingHausdorff} has a stronger form in that not only
is there a Cantor subset with the correct dimension but this subset corresponds to all the subsums of a subsequence of $(s_n)$.

\begin{theorem} \label{thm:Hausdorffdimension}
Suppose that $\dim_H(C_{s,\cD}) = A$ and $dim_P(C_{s,\cD}) = B$.  Then for any $0 \le \alpha \le A$ and $0 \le  \beta \le B$, with $\alpha/A \le \beta/B$, there is a subsequence
$(t_n)$ of $(s_n)$ such that $\dim_H(C_{t,\cD}) = \alpha$ and $\dim_P(C_{t,\cD}) = \beta$.
\end{theorem}

\begin{proof}
We will assume $0< \alpha <A$, $0 < \beta < B$ and leave the details of the endpoint cases for the reader. Choose $n_i$ and $m_i$ to be disjoint sequences of indices such that
\[
    \lim_i   \frac{- n_i \ln(N)}{\ln( s_{n_i})} = A   \quad \mbox{ and } \quad  \lim_i \frac{-m_i \ln(N)}{\ln( s_{m_i} )} = B.
\]
If necessary, we take subsequences in order to assure that  $n_1 \ge 100$, $m_i \ge 2^{n_i}$, and $n_{i+1} \ge 2^{m_i}$.
To obtain the new sequence $t_k$, we remove terms from $s_n$ in segments, each in a ``uniform'' manner with some density $\xi \in (0,1)$.
To explain, suppose the segment is the set of indices $\{q , q+1, \ldots, \ell \} \subset \N$.
Then to uniformly remove terms with density $\xi$ from this segment, we remove all the terms of the form
$q + \lfloor i/\xi \rfloor$ for $i = 0,\ldots, \lfloor \xi (\ell - q) - 1 \rfloor$ (to make sure we do not remove $\ell$).
Note that removing with density $\xi$ is the same as retaining with density $1-\xi$.

From the set of indices $\{1,2,\ldots, n_1 \}$, we remove  terms in a ``uniform'' way with density $1-\frac{\alpha}{A}$.
Then from the set of indices $\{n_1+1, \ldots, m_1 \}$ we remove terms in a ``uniform'' way with density $1 - \frac{\beta}{B}$.
We continue alternating, removing terms with density $1-\frac{\alpha}{A}$ from $\{m_i+1, \ldots, n_{i+1}\}$ and with
density $1 - \frac{\beta}{B}$ from $\{ n_i + 1, \ldots, m_i \}$.
Call the resulting sequence $t_\ell$ where we have $t_ \ell= s_{n}$,
with $\ell = n \Theta(n)$ where $\Theta:\N \to [\alpha/A,\beta/B]$ is a measure of the ``local scaling'' of  the index.
From the construction we have $\Theta(n_j) \approx \alpha/A$, $\Theta(m_j) \approx \beta/B$, $\Theta$ is increasing on $\{ n_i+1,\ldots, m_i\}$ and
decreasing on $\{ m_i+1, \ldots, n_{i+1} \}$.
Further,
\[
    \frac{-\ell \ln(N)}{ \ln( t_\ell )}  = \theta(n) \frac{-n \ln(N)}{\ln( s_n)}.
\]
From here it is straightforward to show that $\liminf \frac{-\ell \ln(N)}{\ln( t_\ell )} = \alpha$ and  also that $\limsup \frac{-\ell \ln(N)}{\ln( t_\ell )} = \beta$, as desired.
The condition $\alpha/A \le \beta/B$ is used to check the new liminf and limsup.  
Since the original sequence satisfies condition (\ref{eq:kappacondition}), it is easy to see that any subsequence will as well.
\end{proof}

Of course, this construction does not guarantee that $C_{t,\cD}$ will satisfy $0 < \cH^\alpha(C_{t,\cD}) < \infty$ even if it has the proper dimension.
Comparing Theorem \ref{thm:comparingHausdorff} with Theorem \ref{thm:Hausdorffdimension}, we trade the ability to specify the
$\cH^t$-measure of the subset with the ability to ensure that the subset is of a particularly nice form, in Theorem
\ref{thm:Hausdorffdimension} being the full set of subsums of some subsequence.  
However, if we assume a bit more on $C_{t,\cD}$ we can obtain a substantially stronger result.

\begin{theorem} \label{thm:alphasets} \mbox{}

\begin{enumerate}

  \item Suppose that $0 < \cH^A(C_{s,\cD}) < \infty$.  Then for any $0 \le a \le A$ there is a subsequence $(t_n)$ of $(s_n)$ such that $0 < \cH^a(C_{t,\cD})< \infty$.

  \item Suppose that $0 < \cP^B(C_{s,\cD}) < \infty$.  Then for any $0 \le b \le B$ there is a subsequence $(t_n)$ of $(s_n)$ such that $0 < \cP^b(C_{t,\cD})< \infty$.

  \item Suppose that $0 < \cH^A(C_{s,\cD}) < \infty$ and $0 < \cP^B(C_{s,\cD}) < \infty$.  
Then for any $0 \le a \le A$ and $0 \le b \le B$ with $a/A \le b/B$, there is a subsequence $(t_n)$ of $(s_n)$ such that 
$0 < \cH^a(C_{t,\cD}) < \infty$ and $0 < \cP^b(C_{t,\cD}) < \infty$.

\end{enumerate}
\end{theorem} 

\begin{proof}
We prove the third statement as it is the most involved.  The other two are similar.  
As in Theorem \ref{thm:Hausdorffdimension} we work with $s_n$ rather than $R_n$.  

Let $m_i, n_i \in \N$ be such that $\lim N^{m_j} s_{m_j}^B = \limsup N^n s_n^B = S$ and $\lim N^{n_j} s_{n_j}^A = \liminf N^n s_n^A = I$.  
In addition, we assume that $n_j < m_j < n_{j+1} < m_{j+1}$, $m_j/n_j \to \infty$, and  $n_{j+1}/m_j \to \infty$.  The two cases $a/A = b/B$ and
$a/A < b/B$ require different techniques and so we do them separately.

\smallskip

\noindent{\emph{Case 1: $a/A = b/B$}}

If $a/A = 1$, then there is nothing to prove. 
%We will assume that $|m_j - n_k| > 2 A/a$ for all $j,k$.
We define our subsequence $(t_n)$ by defining the indexing function $\pi:\N \to \N$ such that $t_n = s_{\pi(n)}$.
Define $\hat{\pi}:\N \to \N$ by $\hat{\pi}(i) = \lfloor (A/a) i \rfloor$. 
 If $m_j, n_j \in \hat{\pi}(\N)$ for all $j$, then we  let $\pi = \hat{\pi}$.
Otherwise, suppose that $m_j \notin \hat{\pi}(\N)$.  
Then $i := \lfloor m_j a/A \rfloor < m_j a/A$, so we define $\pi(i) = m_j$.
We do the same procedure for any $n_j \notin \hat{\pi}(\N)$. 
Since $A/a > 1$ we know that $\hat{\pi}$ is injective.  
If we assume that $|n_j - m_k| > 2 A/a$ for all $j$ and $k$ then $\pi$ is also guaranteed to be injective.
Since $[k,k+A/a+1] \cap \hat{\pi}(\N)$ is nonempty for any $k$, we know that $-1 \le \pi(i) - (A/a) i \le A/a + 1$ or, more useful for us,
\[
%   i - \frac{a}{A} \le \pi(i) \left( \frac{a}{A} \right) \le i + 1 + \frac{a}{A}.
      -1 - \frac{a}{A} + \pi(i) \left( \frac{a}{A} \right) \le i \le \frac{a}{A} + \pi(i) \left( \frac{a}{A} \right).
\]
This means that
\begin{eqnarray*}
    N^i t_i^a &=& N^i s_{\pi(i)}^a \ge N^{-1 - a/A} N^{\pi(i) (a/A)} s_{\pi(i)}^{(a/A) A} = N^{-1 - a/A} \left( N^{\pi(i)} s_{\pi(i)}^A \right)^{a/A}\cr
      &\ge& N^{-1 - a/A} (I-\epsilon)^{a/A} > 0, 
\end{eqnarray*}
for large enough $i$.  Thus $\liminf N^i t_i^a > 0$ and so $\cH^a(C_{t,\cD}) >0$.  By construction, there is a sequence $q_j \in \N$ so that $\pi(q_j) = n_j$ and so
\[
   N^{q_j} t_{q_j}^a  \le N^{a/A} N^{\pi(q_j) (a/A)} s_{n_j}^{(a/A) A} = N^{a/A} \left( N^{n_j} s_{n_j}^A \right)^{a/A} \le N^{a/A} (I + \epsilon)^{a/A} < \infty.
\]
Thus $\cH^a(C_{t,\cD}) < \infty$ as well.  The proof that $0 < \cP^b(C_{t,\cD}) < \infty$ is similar.

\smallskip

\noindent{\emph{Case 2: $a/A < b/B$}}

Let 
\[
   \gamma_0 = \frac{ \frac{B}{b} - 1}{\frac{A}{a} - 1}
\]
and then choose $\delta > 0$ so that $\gamma := \gamma_0 + \delta < 1$.  Define 
\[
     n_j' = \left\lfloor \frac{n_j}{\gamma} \right\rfloor  \quad \mbox{ and } \quad  m_j' = \lfloor \gamma m_j \rfloor
\]
and notice that $n_j' > n_j$ and $m_j' < m_j$.  For notational ease, let
\begin{eqnarray*}
    P_j = m_j - \left\lceil (1-\frac{b}{B}) m_j \right\rceil,  &  \quad &
    P_j' = m_j' - \left\lceil (1-\frac{b}{B}) m_j' \right\rceil, \\[3 pt]
    Q_j = n_j - \left\lfloor(1-\frac{a}{A}) n_j \right\rfloor,  &\mbox{ and }& 
    Q_j' = n_j' - \left\lfloor (1-\frac{a}{A}) n_j' \right\rfloor. \cr
\end{eqnarray*}
Further, let
\[
   d_j = \frac{ P_j' + \lceil (1-\frac{b}{B}) m_j \rceil - (Q_j' + \lfloor (1-\frac{a}{A}) n_j \rfloor )}{ P_j' - Q_j'}
\]
and
\[
   e_j = \frac{ Q_{j+1} + \lfloor (1-\frac{a}{A}) n_{j+1} \rfloor - (P_j + \lceil (1-\frac{b}{B}) m_j \rceil )}{ Q_{j+1} - P_j}.
\]
Define $\pi:\N \to \N$ by
\[
\pi(i)  = 
\begin{cases}
    i + \lfloor (1-\frac{a}{A}) n_j \rfloor & \mbox{ if } Q_j \le i < Q_j' \\
   i + \lceil (1- \frac{b}{B}) m_j \rceil & \mbox{ if } P_j' < i \le P_j \\
   Q_j' + \lfloor (1-\frac{a}{A} ) n_j \rfloor + \lfloor k d_j \rfloor & \mbox{ if } i = Q_j' + k, k = 0,\ldots, P_j' - Q_j' \\
   P_j + \lceil (1-\frac{b}{B}) m_j \rceil + \lceil k e_j \rceil & \mbox{ if } i = P_j + k, k = 1, \ldots, Q_{j+1} - P_j - 1.
\end{cases}
\]
We define $t_i = s_{\pi(i)}$.  The choice of $n_j, m_j$, and  $\gamma$ ensure that $Q_j < Q_j' < P_j' < P_j < Q_{j+1}$.  
We also have $\pi(Q_j) = n_j$ and $\pi(P_j) = m_j$.
It is straightforward but quite tedious to check that $\pi$ is injective and also that 
\begin{equation} \label{eq:ratioestimate}
   \frac{B}{b} \le \frac{\pi(i)}{i} \le \frac{A}{a}
\end{equation} 
for all large $i$.  
We remark that the strict inequality $a/A < b/B$ is necessary in order to show (\ref{eq:ratioestimate}) for the last two cases in the definition of $\pi$.

Thus for $\epsilon > 0$ small and all large enough $i$, we have
\[
    N^i t_i^a = N^i s_{\pi(i)}^a \ge N^{\pi(i) (a/A)} s_{\pi(i)}^{(a/A) A} = \left( N^{\pi(i)} s_{\pi(i)}^A \right)^{a/A} \ge (I-\epsilon)^{a/A} > 0,
\]
and thus $\cH^a(C_{t,\cD}) > 0$. 
For $i = Q_j$, we have $\pi(i) = n_j$ and $Q_j \le (a/A) n_j + 1$ and so
\[
   N^i t_i^a = N^{Q_j} s_{n_j}^a \le N^{(a/A) n_j}  s_{n_j}^a N = N \left( N^{n_j} s_{n_j}^A \right)^{a/A} \le N (I + \epsilon) ^{a/A} < \infty
\]
and so $\cH^a(C_{t,\cD}) < \infty$.  The argument that $0 < \cP^b(C_{t,\cD}) < \infty$ is similar.
\end{proof}

\begin{example}
The simple Example \ref{example:Cantor} will show that in general we cannot find a subsequence which will give a subset of arbitrary measure.
Recall that it was $s_n = 2/3^n$ and $\cD^n = \{0, 1 \}$ for all $n$. % whose associated set $C_{s,\cD}$ is the classical Cantor set.
It is known that $\dim_H C_{s,\cD} = d = \ln(2)/\ln(3)$ and $\cH^d(C_{s,\cD}) = 1$.
The key observation  is that if $t_n$ is a subsequence of $s_n$ constructed by removing only $K$ terms from $s_n$, then $\cH^d(C_{t,\cD}) = 2^{-K}$,
since $C_{s,\cD}$ is the union of $2^K$ disjoint copies of $C_{t,\cD}$ (these copies correspond to the possible subsums of the removed terms).
But this means that it is impossible to find a subsequence $t_n$ with $\cH^d(C_{t,\cD}) = 1/3$.
\end{example}
%%%%%%%%%%%%%%%%%%%%%%%%%%%%%%%%%%%%%%%%%%%

\end{document}